\documentclass[a4paper,12pt]{amsart}
\usepackage[utf8]{inputenc}
\usepackage[english]{babel}
\usepackage[text={160mm,232mm},centering]{geometry}
\usepackage{amsmath,amssymb,amsthm,mathabx,verbatim}
\usepackage[bookmarks=true]{hyperref}
\usepackage{booktabs}
\usepackage{caption}
\usepackage{enumerate}
\usepackage{cleveref}
\usepackage{multirow}
\usepackage{array}
\usepackage{float}
\newcolumntype{C}[1]{>{\centering\arraybackslash}m{#1}}

\numberwithin{equation}{section}

\newtheorem{thm}[equation]{Theorem} 
\newtheorem{lem}[equation]{Lemma} 
\newtheoremstyle{named}{}{}{\itshape}{}{\bfseries}{.}{.5em}{#3}
\theoremstyle{named} 
\theoremstyle{remark}
\theoremstyle{definition}

\numberwithin{table}{section}

\newcommand{\QQ}{\mathbf Q}
\newcommand{\ZZ}{\mathbf{Z}}

\newcommand{\RR}{\mathbf R}
\newcommand{\FF}{\mathbf F}

\begin{document}
\title{Greenberg's conjecture for real quadratic number fields}

\author{Pietro Mercuri}
\email{mercuri.ptr@gmail.com}
\address{Università di Trento, Dipartimento di matematica, Trento 38122, Italy}

\author{Maurizio Paoluzi}
\email{mauriziopaoluzi@gmail.com}
\address{Via Mariano Rampolla 24, Rome 00168, Italy}

\author{René Schoof}
\email{schoof.rene@gmail.com}
\address{Università di Roma ``Tor Vergata'', Dipartimento di matematica, Rome 00133, Italy}


\begin{abstract}
We compute the $3$-class groups $A_n$ of the fields $F_n$ in the cyclotomic $\ZZ_3$-extensions of the real quadratic fields of discriminant $f<100,000$. In all cases the orders of $A_n$ remain bounded as $n$ goes to infinity. This is in agreement with Greenberg's conjecture. 
\end{abstract}

\maketitle

\section{Introduction}

Let $F$ be a totally real number field and let $p$ be a prime. Let
$$
F=F_0\,\,\subset\,\, F_1\,\,\subset\,\, F_2\,\,\subset\,\, \ldots
$$
denote the cyclotomic $\ZZ_p$-extension  of~$F$. The $p$-class group $A_n$ is the $p$-part of the ideal class group of the ring of integers of $F_n$. In his 1971 thesis
Ralph Greenberg conjectured that  $\#A_n$ remains bounded as $n\rightarrow\infty$.  See \cite{GBT,GBA} and~\cite[Conjecture (3.4)]{GBI}. 
This is the so-called ``$\lambda=0$"-conjecture  of Iwasawa theory. In this note we report on
a computation for the prime $p=3$ involving the $30394$ real quadratic fields $\QQ(\sqrt{f})$ of discriminant $f<100,000$. See~\cite{FK,IS,KS} for earlier computations.
As a consequence we obtain the following result.

\begin{thm}\label{thm:greenberg}
Greenberg's conjecture is true for $p=3$ and the real quadratic fields of discriminant~$f<100,000$.
\end{thm}

\medskip\noindent
 For each of the real quadratic fields with discriminant $f$ in the range of our compuation we have computed a certain Galois module $C(f)$, the finiteness of which is equivalent to  Greenberg's conjecture. In this introduction we define
 the module $C(f)$.    In the subsequent sections   we explain our computation and  its results.  
 See \cite{KS} for the algebraic properties of $C(f)$  in the case $f\not\equiv 1\pmod 3$. The slightly different case $f\equiv 1\pmod 3$ is discussed in detail in~\cite{PAO} and in \Cref{sec:1mod3} of this paper.
 
Let $F=\QQ(\sqrt{f})$ be a real quadratic field of discriminant $f$. The Galois module $C(f)$ is defined in terms of cyclotomic units as follows. For $k\ge 1$ let  $\zeta_k$ denote a primitive $k$-th root of unity. 
For  $n\ge 0$  the $n$-th layer in the cyclotomic $\ZZ_3$-extension  of $F$ is
$$
F_n=\QQ(\sqrt{f},\zeta_{3^{n+1}}+\zeta_{3^{n+1}}^{-1}).
$$
 The field $F_n$ is a  subfield of the cyclotomic field $\QQ(\zeta_{3^{n+1}f})$. It is a cyclic  degree $3^n$ extension of $F_0=\QQ(\sqrt{f})$. Its ring of integers $O_n$ contains cyclotomic units.  See~\cite[Section~4]{SINN}. The $3$-part of  the quotient of the unit group $O_n^*$ by the subgroup generated by the cyclotomic units is a finite group denoted by $B_n$.
It is known that the groups $A_n$ and $B_n$ have the same cardinality~\cite[Theorems~4.1 and 5.3]{SINN}.  Therefore
Greenberg's conjecture is true for the field $F$ if and only if $\#B_n$ remains bounded as $n\rightarrow\infty$.

When the discriminant $f$ is not congruent to $1\pmod 3$, we let $C_n$ denote the dual of the group $B_n$ for $n\ge 0$. When $f\equiv1\pmod 3$, we  let  $C_n$ denote the dual of the group $\tilde B_n$. Here $\tilde B_n$ sits in an exact sequence of the form
$$
0\longrightarrow \tilde B_n\longrightarrow B_n\mathop{\longrightarrow}\limits^{\phi_n} \ZZ_3/\log_3\eta_0\ZZ_3.
$$
where for $\epsilon\in O_n^*$ we put $\phi_n(\epsilon)={1\over{3^n}}\log_3(N_n(\epsilon))$. Here  $N_n\colon F_n^*\rightarrow\QQ(\sqrt{f})^*$ is the norm map. Since the $3$-adic logarithm of a generator $\eta_0$ of the group of  cyclotomic units in $\QQ(\sqrt{f})$ is not zero, the rightmost group is a finite cyclic group.
It follows that $[B_n:\tilde B_n]$ and hence the quotient $\#B_n/\#C_n$  is bounded independently of~$n$.
Therefore Greenberg's conjecture is true if and only if $\#C_n$ remains bounded as $n\rightarrow\infty$.
 
By \cite[Lemma 2.1]{KS} and \Cref{sec:1mod3}, the natural maps $B_m\rightarrow B_n$ are injective and the natural maps $C_n\rightarrow C_m$ are surjective for  $n\ge m$. Let  $C(f)$  denote the projective limit of the $C_n$. Then $C(f)$ is a Galois module and hence  in the usual way a module over
the Iwasawa algebra~$\Lambda\cong\ZZ_3[[T]]$.
It follows from the structure  of  the cyclotomic units that $C(f)$   is a cyclic $\Lambda$-module. See~\cite[Theorem~2.4]{KS} and \Cref{sec:1mod3}.  In other words, we have
$$
C(f)\,=\,\mathop{\rm lim}\limits_{\leftarrow} C_n\,\,\cong \Lambda/J,\qquad\text{for some ideal $J\subset\Lambda$.}
$$
The vanishing of the Iwasawa $\mu$-invariant of $\QQ(\sqrt{f})$ means that $J$ contains a monic polynomial and hence that  $C(f)$ is a finitely generated $\ZZ_3$-module. See~\cite{MU}. Greenberg's conjecture affirms that   $C(f)$ is actually \emph{finite}, so that $C_n=C(f)$ for sufficiently large~$n$.
\smallskip\smallskip

We have computed the Galois modules $C(f)$ for $f<100,000$. It took about two weeks on a workstation with Intel processor i5.
We found the following, which  is equivalent to \Cref{thm:greenberg}.

\begin{thm}\label{thm:Cfin}
For $p=3$ and for all discriminants $f<100,000$ the module $C(f)$ is finite.
\end{thm}

\noindent
In most cases the module  $C(f)$ is actually zero.
Indeed,  for only $3359$ out of the $30394$ real quadratic fields considered, $C(f)$ is not zero and, equivalently, $J$ is a proper $\Lambda$-ideal.
This is about $11$\% of all cases. Of these, $2118$ have $J$ equal to the maximal ideal $(3,T)$ of $\Lambda$.
In these cases $C(f)$ has order~$3$.
For the remaining $1241$ discriminants the module $C(f)$ is strictly larger. This is  approximately $4$\% of all cases.

Rather than listing each  ideal~$J$, we indicate in \Cref{sec:not1mod3data,sec:1mod3data} how often ideals of a certain type appear in our computation. The full list of ideals may be of interest in itself and  is available  on GitHub~\cite{GIT}.
In \Cref{sec:not1mod3data,sec:1mod3data} we also single out some discriminants for which the ideal $J$ has a remarkable shape.

\section{The case $f\nequiv 1\pmod 3$}\label{sec:not1mod3}

In this section we give a brief description of the algorithm in the case where the  discriminant $f$  is congruent to $0$ or $2$ modulo~$3$.  
This case is discussed in detail in~\cite{KS}. Let $F=\QQ(\sqrt{f})$ be a real quadratic field of discriminant~$f$.  
Put $f'=f/3$ when $f\equiv 0\pmod 3$ and $f'=f$ when $f\equiv 2\pmod 3$. For $n\ge 0$ the $n$-th layer  $F_n$ in the cyclotomic tower of $F$ is a subfield of
$\QQ(\zeta_{3^{n+1}f'})$. The cyclotomic unit $1-\zeta_{3^{n+1}f'}$ is contained in $\QQ(\zeta_{3^{n+1}f'})$. Put
$$
\eta_n=\mathop{\rm Norm}\limits_{\QQ(\zeta_{3^{n+1}f'})/F_n}(1-\zeta_{3^{n+1}f'})^{\sigma-1}.
$$
Here $\sigma$ is the non-trivial automorphism in ${\rm Gal}(F_n/\QQ_n)\cong {\rm Gal}(\QQ(\sqrt{f})/\QQ)$.

In \cite{KS} it is explained that the Galois module generated by $\eta_n$ is free of rank~$1$ over~$\ZZ[G_n]$. Here $G_n$ denotes~${\rm Gal}(F_n/F_0)$.
This  implies that the Galois module $C(f)$  described in the introduction is a \emph{cyclic} module over the Iwasawa algebra $\Lambda=\mathop{\rm lim}\limits_{\leftarrow}\ZZ_3[G_n]\cong\ZZ_3[[T]]$. 
So we have  $C(f)=\Lambda/J$ for some $\Lambda$-ideal~$J$.  
For $n\ge 0$ we put $\omega_n(T)=(1+T)^{p^n}-1$ and we write $(\omega_n)$ for the $\Lambda$-ideal generated by it.  In \cite{KS} it is explained that in this case  
we have 
$$
C_n=C(f)/\omega_nC(f)=\Lambda/(J+(\omega_n)),\qquad\text{ for all~$n\ge 0$.}
$$ 
The Galois module $C(f)$ is finite if and only if $\omega_nC(f)=0$ and hence $C(f)=C_n$ for some~$n\ge 0$. By Nakayama's lemma this happens  if and only if $J+(\omega_n)=J+(\omega_{n+1})$ for some~$n\ge 0$.
This observation leads to the following algorithm. For   $n=0,1,2,\ldots$, we compute the shrinking ideals $J+(\omega_n)$ 
until  we find that $J+(\omega_n)=J+(\omega_{n+1})$.

Our method for computing the ideals $J+(\omega_n)$ runs as follows. 
For a given $n$ we first calculate a lot of  elements in the ideal.
As is explained in~\cite{KS}, this involves  calculations with cyclotomic units modulo primes $r\equiv 1\pmod{f'3^{n'}}$ for suitable $n'>n$. This  leads to an \emph{upper bound} for 
$\Lambda/(J+(\omega_n))$. To obtain a \emph{lower bound} we employ a method due to G. and M.-N. Gras~\cite{GG}. This involves calculations with high precision approximations  of the cyclotomic units in $F_n\otimes\RR$. See also \cite[Section 4]{KS}. 
Clearly, when the upper and lower bounds agree, we have determined $J+(\omega_n)$ and hence $C_n=\Lambda/(J+(\omega_n))$. 

The calculation of the lower bound for $C_n$ becomes very time consuming and takes a lot of memory as $n$ grows. 
This is caused by the high precision computations with units in cyclotomic fields of seven digit conductors and  degrees in the hundreds.
 In fact, for most discriminants $f$ it becomes infeasible when $n$ exceeds~$2$. Fortunately, for most $f$ 
 we find that $J+(\omega_n)=J+(\omega_{n+1})$ and hence $C(f)=C_n$ for $n\le 2$.

In the rare cases where we need to consider $J+(\omega_n)$ for  $n\ge 3$, it is still feasible to compute 
the upper bound in the sense  that we can easily calculate  a lot of elements in the ideal~ $J+(\omega_n)$.
 An application of the Cebotarev density theorem suggests that these elements probably \emph{generate} 
 $J+(\omega_n)$, so that our upper bound is actually \emph{equal} to the lower bound, but we have no rigorous proof of this.

Fortunately, we can still  rigorously prove that $C(f)=\Lambda/J$ is finite and thus confirm Greenberg's conjecture
even when we cannot use our algorithm to compute  lower bounds for $\Lambda/(J+(\omega_n))$.
It suffices to have an upper bound for $n$ and a lower bound for \emph{some} $m\le n$ to which the following lemma applies.  In the range of our computations this always works out with~$n\ge m=2$.

\begin{lem}\label{lem:finiteness}
Let $M$ be a finitely generated $\Lambda$-module.
Suppose that  for certain  integers $n\ge m\ge 0$ and $b\ge a\ge 0$ we have
$$
\#M/\omega_mM\ge p^a\quad\hbox{and}\quad \#M/\omega_nM\le p^b.
$$
If $b-a<n-m$, then $\omega_n M = 0$. In particular, if $M/\omega_n M$ is finite, so is $M$.
\end{lem}
\begin{proof}
In the filtration 
$$
\omega_nM \quad\subset\quad \omega_{n-1}M \quad\subset\ldots\subset\quad \omega_{m+1}M \quad\subset\quad \omega_mM
$$
there are $n-m$ inclusions. We have inequalities
$$
\#(\omega_nM/\omega_mM)={{\#M/\omega_nM}\over{\#M/\omega_m M}}\le p^{b-a}<p^{n-m}.
$$
It follows that one of the inclusions must be an equality.
So we have $\omega_{k+1}M=\omega_{k}M$ for some $k=m,\ldots,n-1$. Then $x=\omega_{k+1}/\omega_k$ is an element of the maximal ideal of $\Lambda$ that has the property that $x\omega_kM=\omega_kM$.  Nakayama's lemma  implies then $\omega_kM=0$. It follows that $\omega_nM$ is zero, as required.
\end{proof}

\section{Numerical data for discriminants $f\nequiv 1\pmod 3$}\label{sec:not1mod3data}

\subsection*{Case $f\equiv 0\pmod 3$.}\hfill\smallskip\\
There are $7606$ real quadratic fields with discriminant $f\equiv 0\pmod 3$ and $f<100,000$.
For precisely $769$ of them the Galois module $C(f)=\Lambda/J$ is not zero. This is approximately~$10\%$.
For $513$ discriminants $J$ is equal to the maximal ideal $(3,T)$ of $\Lambda$. For the remaining $256$ discriminants 
 $J$ is strictly smaller. \Cref{table1} contains some data.

\begin{table}[h!]  
\renewcommand*{\arraystretch}{1.2}
\caption{The modules $\Lambda/J$ for $f\equiv 0\pmod 3$.}
\begin{center}
\begin{tabular}{r|r|r|r|r}
$n$\phantom{*}&$T\phantom{^2}$&$T^2$&$T^3$&Total \\
\hline
$0$\phantom{*}&\phantom{*}$536$\phantom{*}&$0$\phantom{*}&$0$\phantom{*}&$536$ \\
$1$\phantom{*}&$112$\phantom{*}&\phantom{*}\phantom{*}$50$\phantom{*}&\phantom{*}\phantom{*}\phantom{*}$2$\phantom{*}&$164$ \\
$2$\phantom{*}&$35$\phantom{*}&$7$\phantom{*}&$2$\phantom{*}&$44$ \\
$3$\phantom{*}&$15$*&$0$\phantom{*}&$0$\phantom{*}&$15$ \\
$4$\phantom{*}&$5$*&$1$*&$0$\phantom{*}&$6$ \\
$5$\phantom{*}&$2$*&$0$\phantom{*}&$0$\phantom{*}&$2$ \\
$6$\phantom{*}&$2$*&$0$\phantom{*}&$0$\phantom{*}&$2$ \\
\hline
&$707$\phantom{*}&$58$\phantom{*}&$4$\phantom{*}&$769$
\end{tabular}
\end{center}
\label{table1}
\end{table}

\noindent
The rows of \Cref{table1} correspond to the \emph{level of stabilization}~$n$.  This means that
$n$ is the smallest integer for which 
the ideals $J+(\omega_n)$ and $J+(\omega_{n+1})$ are equal and hence $J=J+(\omega_n)$. In particular, we have~$\Lambda/J=C(f)=C_n$.
The number $n$ is also the smallest for which $\omega_n=(1+T)^{3^n}-1$ is in~$J$.  Equivalently, $3^n$ is the order of $1+T$ in the multiplicative group $(\Lambda/J)^*$.

The columns are indexed by the symbols $T^k$ for $k=1,2,\ldots$. 
The entry in the $n$-th row and the $T^k$-column is the number of discriminants for which
the  level of stabilization is $n$,  and  the image of $J$ in the ring $\FF_3[[T]]$ is the ideal $(T^k)$.
Since $\omega_n$ is congruent to $T^{3^n}$ modulo $3$, the $(n, T^k)$-entry is zero whenever $k>3^n$.
In particular, in the row corresponding to $n=0$, all entries  with $k>1$ are zero.
 
In the first column  we count the discriminants for which the ideal $J$ is of the form $J=(T-a,b)$ for certain $a,b\in\ZZ$. 
For $536$ discriminants we have $a=0$ and there is stabilization at level $n=0$. 
This means that $\#C_0=\#C_1$ or, equivalently $\#A_0=\#A_1$.
The  discriminants for which $J$ is equal to the maximal ideal of $\Lambda$ are included  here.
This entry was checked by computing the class numbers of the fields $F_0$ and $F_1$ of degrees $2$ and $6$ respecively using a few lines of PARI/GP~\cite{PARI} code.
For the other entries in the first column, we have $a\not\in b\ZZ_3$ and stabilization occurs at level~$n= v_3(b/a)$.

An asterisk indicates that we do not have a rigorous lower bound for $C(f)$ for some of the discriminants appearing in this entry.  However, our upper bound is very likely to be sharp, so that almost certainly $C(f)$ is isomorphic to $\Lambda/J$.
In each case \Cref{lem:finiteness} was applied to prove Greenberg's conjecture. 
The  $62$ cases appearing in the second and third columns were dealt with using the polynomial arithmetic of Magma~\cite{MAG}.
We single out nine discriminants $f$ for special mention.

\begin{table}[h!]  
\renewcommand*{\arraystretch}{1.2}
\caption{Exotic Galois modules for $f\equiv 0\pmod 3$.}
\begin{center}
\begin{tabular}{c|l|r|r}
$f$&\multicolumn{1}{c|}{$J$}&$n$&$T^k$ \\
\hline
$31989$&$(T-996,2187)$&$6$&$T\phantom{^2}$\\
$38424$&$(T+261,2187)$&$5$&$T\phantom{^2}$\\
$59061$&$(T^2+3T-9,81)$&$4$&$T^2$\\
$60513$&$(T^3+3,3T,9)$&$2$&$T^3$\\
$61629$&$(T^3,3)$&$1$&$T^3$\\
$69117$&$(T+69,729)$&$5$&$T\phantom{^2}$\\
$71049$&$(T^3,3)$&$1$&$T^3$\\
$76584$&$(T^3+3,3T,9)$&$2$&$T^3$\\
$95385$&$(T-2988,6561)$&$6$&$T\phantom{^2}$
\end{tabular}
\end{center}
\label{table2}
\end{table}

\subsection*{Case $f\equiv 2\pmod 3$.}\hfill\smallskip\\
There are $11394$ real quadratic fields with discriminant $f\equiv 2\pmod 3$ and $f<100,000$.
For precisely $1250$ of them the Galois module $C(f)=\Lambda/J$ is not zero. This is
approximately $11\%$ of all discriminants.
For $781$ discriminants $J$ is equal to the maximal ideal $(3,T)$ of $\Lambda$. For the remaining $469$ discriminants 
 $J$ is strictly smaller.  This is about $4\%$ of all cases. \Cref{table3} contains some data.

\begin{table}[h!]  
\renewcommand*{\arraystretch}{1.2}
\caption{The modules $\Lambda/J$ for $f\equiv2\pmod 3$.}
\begin{center}
\begin{tabular}{r|r|r|r|r|r}
$n$\phantom{*}&$T\phantom{^2}$&$T^2$&$T^3$&$T^4$&Total \\
\hline
$0$\phantom{*}&$827$\phantom{*}&$0$\phantom{*}&\phantom{*}\phantom{*}$0$\phantom{*}&\phantom{*}\phantom{*}$0$\phantom{*}&$827$\\
$1$\phantom{*}&$158$\phantom{*}&\phantom{*}$87$\phantom{*}&$8$\phantom{*}&$0$\phantom{*}&$253$\\
$2$\phantom{*}&$101$\phantom{*}&$7$\phantom{*}&$4$\phantom{*}&$1$\phantom{*}&$113$\\
$3$\phantom{*}&$36$*&$2$*&$0$\phantom{*}&$0$\phantom{*}&$38$\\
$4$\phantom{*}&$13$*&$1$*&$0$\phantom{*}&$0$\phantom{*}&$14$\\
$5$\phantom{*}&$4$*&$0$\phantom{*}&$0$\phantom{*}&$0$\phantom{*}&$4$\\
$6$\phantom{*}&$1$*&$0$\phantom{*}&$0$\phantom{*}&$0$\phantom{*}&$1$\\
\hline
&\phantom{*}$1140$\phantom{*}&$97$\phantom{*}&$12$\phantom{*}&$1$\phantom{*}&$1250$
\end{tabular}
\end{center}
\label{table3}
\end{table}

\noindent
The interpretation of the entries of the table is the same as  in the case $f\equiv 0\pmod 3$.
The $781$ discriminants with $J=(3,T)$ are included in the entry with $n=0$  of the first column.
Also in this case the discriminants in the first column were taken checked using 
a few lines of PARI/GP code.
The other $110$ cases were dealt with using the polynomial arithmetic of Magma.
We single out nine discriminants for special mention.

\begin{table}[h!]  
\renewcommand*{\arraystretch}{1.2}
\caption{Exotic Galois modules for $f\equiv2\pmod 3$.}
\begin{center}
\begin{tabular}{c|l|r|r}
$f$&\multicolumn{1}{c|}{$J$}&$n$&$T^k$ \\
\hline
$14165$&$(T-255,729)$&$5$&$T\phantom{^2}$\\
$16673$&$(T+462,2187)$&$6$&$T\phantom{^2}$\\
$29165$&$(T-282,729)$&$5$&$T\phantom{^2}$\\
$47633$&$(T^2-9,3T-90,243)$&$4$&$T^2$\\
$51809$&$(T^2+18,3T-18,81)$&$3$&$T^2$\\
$71921$&$(T^2+18,3T+18,81)$&$3$&$T^2$\\
$76604$&$(T+294,729)$&$5$&$T\phantom{^2}$\\
$90005$&$(T+15,729)$&$5$&$T\phantom{^2}$\\
$98105$&$(T^4+3,3T,9)$&$2$&$T^4$
\end{tabular}
\end{center}
\label{table4}
\end{table}

\section{The case $f\equiv 1\pmod 3$.}\label{sec:1mod3}

As before we write $F=\QQ(\sqrt{f})$ and $F_n$ for  the $n$-th layer in the cyclomic $\ZZ_3$-extension of $F=F_0$.
When the discriminant $f$ is congruent to $1\pmod 3$, our method to compute the Galois module $C(f)$ is the same, but the details are slightly different.  See~\cite{PAO}.
The differences are caused by the fact that the Galois module generated by the cyclotomic unit $\eta_n$ is \emph{not free}
over the ring $\ZZ[G_n]$ when $f\equiv 1\pmod 3$. Here $\eta_n$ is defined in \Cref{sec:not1mod3} and $G_n$ denotes~${\rm Gal}(F_n/F_0)$.
Indeed, in this case we have $N_n\eta_n=1$, where $N_n$ is the norm map $F_n^*\rightarrow F^*$. 
When $f\equiv 1\pmod 3$, the 
Galois module ${\rm Cyc}_n$ of cyclotomic units in $F_n$ on which $\sigma$ acts as $-1$,  is a direct product of the  submodules generated by $\eta_n$ and $\eta_0$.
Here $\eta_0$ is the cyclotomic unit in $F_0$. It generates a group isomorphic to $\ZZ$ with trivial Galois action. On the other
hand, the Galois module $\langle\eta_n\rangle$ generated by $\eta_n$ is free of rank~$1$ over the ring~$\ZZ[G_n]/(N_n)$. See~\cite{SINN}.

The  submodule $\tilde{B}_n$ of $B_n$ that was defined in the introduction, is isomorphic to $O_{n,1}^*/\langle\eta_n\rangle$.
 Here $O_{n,1}^*$ denotes the part of the kernel of the norm map $N_n:O_n^*\rightarrow O_0^*$ on which $\sigma$ acts as $-1$.
The Galois modules  $O_{n,1}^*$, $\langle\eta_n\rangle$ and $\tilde{B}_n$ are killed by the norm $N_n$ and are hence $\ZZ[G_n]/(N_n)$-modules.
Since $\langle\eta_n\rangle$ is free of rank~$1$,  it is more convenient to deal with  $\tilde{B}_n$ rather than with  $B_n$ itself.  For instance, from the exactness of
 the sequence of $\ZZ[G_n]/(N_n)$-modules 
 $$
 0\longrightarrow  \langle\eta_n\rangle\longrightarrow O_{n,1}^*\longrightarrow \tilde{B}_n\longrightarrow 0
 $$
one deduces that the natural map $\tilde{B}_m\rightarrow \tilde{B}_n$ identifies $\tilde{B}_m$ with the kernel of the endomorphism $\omega_m'$ of $\tilde{B}_n$
 for $m\le n$. Here we put $\omega'_m=\omega_m/T$.  In particular, we have $\omega_0'=1$ and $C_0=0$.
  It follows that the Galois module 
 $C(f)$ is isomorphic to $\Lambda/J$ for some ideal $J$ and $C_n=C(f)/\omega'_nC(f)=\Lambda/(J+(\omega'_n))$ for all~$n\ge 0$.  

Our strategy is the one explained in \Cref{sec:not1mod3}: for each $n=1,2,\ldots,$ we compute the shrinking ideals $J+(\omega'_n)$ until  we find $J+(\omega'_n)=J+(\omega'_{n+1})$, in which case Nakayama's lemma implies that $J=J+(\omega'_n)$ and hence $C(f)=C_n$ and we are done.
When $f\equiv 1\pmod 3$ the issues with upper bounds and lower bounds are similar to the ones described in \Cref{sec:not1mod3} for $f\not\equiv 1\pmod 3$.
In particular, we can still prove  that $C(f)=\Lambda/J$ is finite in each case in the range of our computations. When the lower bound is not available for some $n\ge 3$, we invoke \Cref{lem:finiteness} with $\omega_m$ and $\omega_n$ replaced by $\omega_m'$ and $\omega_n'$ respectively.
 
 \medskip
 It is not relevant for our algorithm and computations, but in the rest of this section we analyze the cokernel of the inclusion map $\tilde{B}_n\hookrightarrow B_n$. 
 For $n\ge 0$, let $U_n$ denote the part of the unit group $(O_n\otimes\ZZ_3)^*$ on which $\sigma$  acts as~$-1$.
Since $\tilde{B}_n$ is the kernel of the map ${B}_n\longrightarrow U_0/\langle\eta_0\rangle$
induced  by $\epsilon\mapsto {\root{p^n}\of{N_n\epsilon}}$ for $\epsilon\in O_n^*$,  the quotient 
$B_n/\tilde{B}_n$  is isomorphic to a subgroup of the cyclic group $U_0/\langle\eta_0\rangle$ and is hence bounded 
independently of $n$. This can be made more precise.

The group  $N_nU_n$ is equal to the subgroup $U_0^{p^n}$ of $U_0$. It follows that $N_nO_n^*$ is contained in $U_0^{p^n}$.
Put
$$
{\root{p^n}\of {N_nO_n^*}}=\{u\in U_0:u^{p^n}\in N_nO_n^*\}.
$$
For every $n\ge 0$ we have inclusions
$$
N_n{O_n^*}^p=N_{n+1}O_n^*\subset N_{n+1}O_{n+1}^*\subset N_nO_n^*.
$$
It follows that we have a filtration
$$
O_0^*\subset\ldots \subset {\root{p^n}\of {N_nO_n^*}}\subset {\root{p^{n+1}}\of {N_{n+1}O_{n+1}^*}}\subset\ldots\subset U_0,
$$
with successive subquotiens of order~$1$ or $p$.  The fact that  $N_n{\rm Cyc}_n$ is equal to $\langle\eta_0^{p^n}\rangle$ gives rise to  the isomorphisms
$$
B_n/\tilde{B}_n\cong {N_nO_n^*}/\langle\eta_0^{p^n}\rangle \cong  {\root{p^n}\of {N_nO_n^*}}/\langle\eta_0\rangle.
$$
This leads to the following filtration
$$
O_0^*/\langle\eta_0\rangle\subset\ldots \subset B_n/\tilde{B}_n\subset B_{n+1}/\tilde{B}_{n+1}\subset\ldots\subset U_0/\langle\eta_0\rangle.
$$
with successive subquotiens of order~$1$ or $p$. 
The leftmost group is cyclic of order $h_0=\#A_0$ and the rightmost group has order $\log_3\eta_0$.  
Writing  $\epsilon_0$ for  a fundamental unit of $F=\QQ(\sqrt{f})$, there are $v_3\log_3\epsilon_0$ distinct steps in this filtration.
By Nuccio~\cite{NU}
we have $ B_n/\tilde{B}_n=U_0/\langle\eta_0\rangle$ when $n$ is sufficiently large.

\section{Numerical data for  discriminants $f\equiv 1\pmod 3$.}\label{sec:1mod3data}

There are $11394$ real quadratic fields with discriminant $f\equiv 1\pmod 3$ and $f<100, 000$.
For precisely $1340$ of them the  module  $C(f)$ is not zero. This is approximately $12\%$ of all discriminants.
For $824$ discriminants $J$ is equal to the maximal ideal $(3,T)$ of $\Lambda$. For the remaining $516$ discriminants  the ideal  $J$ is strictly smaller. 
This is $4.5\%$ of all cases.

The mathematics is a bit different in this case. First of all, the groups $A_0$, $B_0$ are irrelevant
for our computations and we have $C_0=0$. In addition, every module $C_n$ is a cyclic module over the ring $\Lambda/(\omega_n)$ that is killed by $\omega_n'$.
In particular, $C_1$ is a cyclic module over the discrete valuation ring $\Lambda/(\omega'_1)$, where $\omega'_1=\omega_1/T=T^2+3T+3$.  Since $T$ is a uniformizer of the ring $\Lambda/(\omega'_1)$, the module  $C_1$ is isomorphic to $\Lambda/(T^2+3T+3,T^k)$ for some~$k\ge 0$.

\begin{table}[H]  
\renewcommand*{\arraystretch}{1.2}
\caption{The modules $\Lambda/J$ for $f\equiv1\pmod 3$.}
\begin{center}
\begin{tabular}{r|r|r|r|r|r|r}
$n$\phantom{*}&$T\phantom{^2}$&$T^2$&$T^3$&$T^4$&$T^5$&Total \\
\hline
$1$\phantom{*}&$824$\phantom{*}&$79$\phantom{*}&$0$\phantom{*}&$0$\phantom{*}&$0$\phantom{*}&$903$\\
$2$\phantom{*}&$249$\phantom{*}&$18$\phantom{*}&$8$\phantom{*}&$1$\phantom{*}&$0$\phantom{*}&$276$\\
$3$\phantom{*}&$88$\phantom{*}&$7$\phantom{*}&$1$\phantom{*}&$0$\phantom{*}&$1$\phantom{*}&$97$\\
$4$\phantom{*}&$47$*&$3$*&$0$\phantom{*}&$0$\phantom{*}&$0$\phantom{*}&$50$\\
$5$\phantom{*}&$9$*&$0$\phantom{*}&$1$*&$0$\phantom{*}&$0$\phantom{*}&$10$\\
$6$\phantom{*}&$2$*&$0$\phantom{*}&$0$\phantom{*}&$0$\phantom{*}&$0$\phantom{*}&$2$\\
$7$\phantom{*}&$2$*&$0$\phantom{*}&$0$\phantom{*}&$0$\phantom{*}&$0$\phantom{*}&$2$\\
\hline
&\phantom{*}$1221$\phantom{*}&\phantom{*}$107$\phantom{*}&\phantom{*}$10$\phantom{*}&\phantom{*}$1$\phantom{*}&\phantom{*}$1$\phantom{*}&$1340$
\end{tabular}
\end{center}
\label{table5}
\end{table}

We single out eleven discriminants for special mention.

\begin{table}[h!]  
\renewcommand*{\arraystretch}{1.2}
\caption{Exotic Galois modules for $f\equiv1\pmod 3$.}
\begin{center}
\begin{tabular}{c|l|r|r}
$f$&\multicolumn{1}{c|}{$J$}&$n$&$T^k$ \\
\hline
$15217$&$(T^4+3,3T,9)$&$2$&$T^4$\\
$30904$&$(T^3-27,3T-63,243)$&$5$&$T^3$\\
$39256$&$(T+621,2187)$&$7$&$T\phantom{^2}$\\
$40441$&$(T^2,9T-27,81)$&$4$&$T^2$\\
$44053$&$(T+348,729))$&$6$&$T\phantom{^2}$\\
$57832$&$(T^2+27,3T-27,81)$&$4$&$T^2$\\
$71821$&$(T^3+18,3T+9,27)$&$3$&$T^3$\\
$78037$&$(T-849,2187)$&$7$&$T\phantom{^2}$\\
$80056$&$(T^5+9T+9,3T^2+18,27)$&$3$&$T^5$\\
$81769$&$(T^2+18,3T+9,81)$&$4$&$T^2$\\
$96712$&$(T-30,729)$&$6$&$T\phantom{^2}$
\end{tabular}
\end{center}
\label{table6}
\end{table}

By Nakayama's lemma the ideal $J$ contains a monic polynomial of degree~$1$ if and only if the ideal $(T^2+3T+3,T^k)$ does.
If $J$ is a proper ideal, this happens precisely when $k=1$, in which case $C_1$ is isomorphic to the order $3$ module $\Lambda/(3,T)$. These cases appear in the first column of \Cref{table5} and were computed using PARI/GP. Their ideals $J$ are of the form $(T-a,b)$ with level of stabilization equal to $v_3(b)$.  
In particular, the first entry contains the $824$ discriminants for which $J$ is equal to the ideal~$(3,T)$.
The $119$ entries in the remaining columns of \Cref{table5} were taken care of using Magma's polynomial arithmetic.

\end{document}